\documentclass[12pt, reqno]{amsart}
\usepackage{amsmath, amsthm, amscd, amsfonts, amssymb, mathtools, color, hyperref}
\usepackage[dvipsnames]{xcolor}

\newtheorem{theorem}{Theorem}[section]

\theoremstyle{definition}

\theoremstyle{remark}
\newtheorem{remark}[theorem]{Remark}
\numberwithin{equation}{section}

\begin{document}
\setcounter{page}{1}

\title[Hoffman-Wielandt type inequality]
{Hoffman-Wielandt type inequality for block companion matrices of certain 
matrix polynomials}
	
\author[Pallavi, Shrinath and Sachindranath]{Pallavi. B, Shrinath Hadimani and 
Sachindranath Jayaraman}
\address{School of Mathematics\\ 
Indian Institute of Science Education and Research Thiruvananthapuram\\ 
Maruthamala P.O., Vithura, Thiruvananthapuram -- 695 551, Kerala, India.}
\email{(pallavipoorna20, srinathsh3320, sachindranathj)@iisertvm.ac.in, 
sachindranathj@gmail.com}

	

\begin{abstract} 
Matrix polynomials with unitary/doubly stochastic coefficients form the subject 
matter of this manuscript. We prove that if $P(\lambda)$ is a  
quadratic matrix polynomial whose coefficients are either unitary matrices 
or doubly stochastic matrices, then under certain conditions on these coefficients, 
the corresponding block companion matrix $C$ is diagonalizable. Consequently, if 
$Q(\lambda)$ is another quadratic matrix polynomial with corresponding block companion 
matrix $D$, then a Hoffman-Wielandt type inequality holds for the block companion matrices 
$C$ and $D$. 
\end{abstract}
	
\maketitle
	
\section{Introduction}\label{sec-1}

\medskip

We work either over the field $\mathbb{C}$ of complex numbers or over the field 
$\mathbb{R}$ of real numbers. The vector space of $n \times n$ matrices over 
$\mathbb{C}$ (resp., $\mathbb{R}$) is denoted by $M_n(\mathbb{C})$ (resp., 
$M_n(\mathbb{R})$). The notations $||\cdot||_2$ and $||\cdot||_F$ will denote 
respectively the spectral norm and the Frobenius norm of a square matrix. Any other 
notation used in the manuscript will be introduced as and when needed. 
	
\medskip
If $A$ and $B$ are two $n \times n$ commuting normal matrices with eigenvalues 
$\lambda_1, \ldots, \lambda_n$ and $\mu_1, \ldots, \mu_n$, respectively, 
then 
\begin{equation}\label{eqn-1}
\displaystyle \sum_{i=1}^{n} |\lambda_i - \mu_i|^2 = ||A-B||^2_F.
\end{equation} 
What happens when the matrices do not commute has given rise to several 
interesting questions. One of the classical and well known inequality in matrix analysis 
is the Hoffman-Wielandt inequality which says that if $A$ and $B$ are two $n \times n$ 
normal matrices with eigenvalues $\lambda_1, \ldots, \lambda_n$ and 
$\mu_1, \ldots, \mu_n$, respectively given in some order, then there exists a permutation 
$\pi$ of $\{1, \ldots, n\}$ such that 
\begin{equation}\label{HW-ineq-1}
\displaystyle \sum_{i=1}^{n} |\lambda_i - \mu_{\pi(i)}|^2 \leq ||A-B||^2_F 
\end{equation}
(see for instance Theorem $6.3.5$, \cite{Horn-Johnson}). The Hoffman-Wielandt 
inequality itself is a particular case of a more general problem in matrix analysis 
called the Wielandt-Mirsky conjecture (see for instance \cite{Cong-Trinh Le}). 
Possible generalizations of the Hoffman-Wielandt inequality by allowing one or both 
of the matrices to be non-normal have been studied in the literature. The most general 
form assumes that one of the matrices is diagonalizable and the other one arbitrary. 
A brief description of these may be found in \cite{Ikramov-Nesterenko}. As in 
\cite{Ikramov-Nesterenko}, we call the general form as the Hoffman-Wielandt type inequality, 
which is stated below (see Theorem \ref{starting theorem}). Other references on eigenvalue perturbation results are \cite{Bhatia, Bhatia-Elsner, Elsner, Elsner-Friedland} and the 
references cited therein. Notice that the block companion matrix of a matrix polynomial (see 
Section \ref{sec-2} for the definition) whose coefficients are normal matrices need not be a 
normal matrix. It turns out that the inequality \eqref{HW-ineq-1} fails to hold in general 
for block companion matrices of matrix polynomials, whose coefficients are normal matrices. 
More on this will be elaborated in Section \ref{sec-3}. This motivates us to consider the 
Hoffman-Wielandt type inequality for the block companion matrices of matrix polynomials 
in this manuscript.
	
\begin{theorem}\label{starting theorem}
(Theorem $4$, \cite{Ikramov-Nesterenko}) Let $A$ be a diagonalizable matrix of 
order $n$ and $B$ be an arbitrary matrix of order $n$, with eigenvalues 
$\alpha_1, \alpha_2, \ldots, \alpha_n$ and $\beta_1, \beta_2, \ldots, \beta_n$, 
respectively. Let $X$ be a nonsingular matrix 
whose columns are eigenvectors of $A$. Then, there exists a permutation $\pi$ of the 
indices $1,2, \ldots, n$ such that 
$\displaystyle \sum_{i=1}^{n} |\alpha_i - \beta_{\pi(i)}|^2 \leq 
||X||^2_2 ||X^{-1}||^2_2 ||A-B||^2_F$.
\end{theorem}

The number $||X||_2 ||X^{-1}||_2$ is the spectral condition number of $X$ and is usually 
denoted by $\kappa(X)$. For a nonsingular matrix $X \in M_n(\mathbb{C}), \ \kappa(X) 
= \displaystyle \frac{\sigma_{\text{max}} }{\sigma_{\text{min}}}$, the ratio of the largest 
and the smallest singular values of $X$. In Section \ref{sec-4}, we estimate $\kappa(.)$ for 
a particular block Vandermonde matrix $X$ which diagonalizes the block companion matrix 
of a given matrix polynomial.

\subsection{Matrix Polynomials}\hspace*{\fill}
\label{sec-2}

\medskip

We present the necessary definitions and preliminaries on matrix polynomials 
in this section. An $n \times n$ matrix polynomial of degree $m$ is a mapping $P$ from 
$\mathbb{C}$ to $M_n(\mathbb{C})$, defined by $P(\lambda) = \displaystyle \sum_{i=0}^{m} 
A_i\lambda^i$, where $A_i \in M_n(\mathbb{C})$ and $A_m \neq 0$. Matrix polynomials arise 
in several areas of applied mathematics. A good source of reference on matrix polynomials 
is the monograph by Gohberg, Lancaster and Rodman \cite{GLR-Book-2}. It is well known that 
any matrix polynomial is equivalent to a diagonal form  
$\text{diag}(d_1,d_2, \ldots, d_r, 0, 0, \ldots, 0)$ through unimodular matrix polynomials. 
This factorization is called the Smith normal form of $P(\lambda)$. The polynomials $d_j$ 
are called the invariant factors of $P(\lambda)$ and $d_j | d_{j+1}$ 
for every $j$. The matrix polynomial $P(\lambda)$ is said to be regular if 
$\text{det} P(\lambda)$ is nonzero as a polynomial in $\lambda$. In this case $r= n$. For 
a regular matrix polynomial $P(\lambda)$, the polynomial eigenvalue problem (PEP) seeks to 
find a scalar $\lambda_0$ and a nonzero vector $v$ such that $P(\lambda_0)v = 0$. 
Equivalently, for a given regular matrix polynomial $P(\lambda), \lambda_0 \in \mathbb{C}$ 
is an eigenvalue, if $\text{det}P(\lambda_0) = 0$. The nonzero vector $v \in \mathbb{C}^n$  satisfying the equation $P(\lambda_0)v=0$ is called an 
eigenvector of $P(\lambda)$ corresponding to an eigenvalue $\lambda_0$. Moreover, 
$\lambda_0 = 0$ is an eigenvalue of $P(\lambda)$ if and only if $A_0$ is singular. We say 
$\infty$ is an eigenvalue of $P(\lambda)$, when $0$ is an eigenvalue of the reverse matrix polynomial $\widehat{P} (\lambda):= \lambda^m P(\frac{1}{\lambda}) = A_0 \lambda^m + A_1 \lambda^{m-1} + \cdots + A_{m-1} \lambda + A_m$. Notice that for an 
$n \times n$ matrix polynomial of degree $m \geq 1$, there are at most $mn$ number of 
eigenvalues.

\medskip
Given an $n\times n$ matrix polynomial $P(\lambda)$ of degree $m$ with 
nonsingular leading coefficient, one can introduce a monic matrix polynomial corresponding 
to $P(\lambda)$ as follows: 
For $i = 0, \ldots, m-1$, let $U_i = A_m^{-1}A_i$. Define a monic matrix polynomial 
$P_U(\lambda):= I \lambda^m + U_{m-1} \lambda^{m-1} + \cdots + U_1 \lambda + U_0$ 
so that $P(\lambda) = A_m P_U(\lambda)$. The matrix polynomials $P$ and $P_U$ have the same 
eigenvalues. Moreover, the eigenvalues of $P_U$ are the same as that of the eigenvalues 
of the corresponding $mn \times mn$ block companion matrix 
$C:= \begin{bmatrix}
	 0 & I & 0 & \cdots &  0\\
	 0 & 0 & I & \cdots & 0 \\
	\vdots & \vdots &  \vdots & \ddots & \vdots\\
	 0 & 0 & 0 & \cdots & I\\
	-U_0 & -U_1 & -U_2 & \cdots & -U_{m-1}
	\end{bmatrix}$ (see \cite{Higham-Tisseur} for details). A nonzero 
vector $v \in \mathbb{C}^n$ is an eigenvector of $P(\lambda)$ corresponding to an 
eigenvalue $\lambda_0$ if and only if the vector 
$\begin{bmatrix*}[c]
v \\ \lambda_0 v \\ \vdots \\ \lambda_0^{m-1} v
\end{bmatrix*} \in \mathbb{C}^{mn}$ is an eigenvector of $C$ corresponding to the 
eigenvalue $\lambda_0$. Note that any eigenvector of $C$ is of this form. \\

\noindent
{\bf Assumptions:} Throughout this manuscript, we only work with matrix polynomials whose 
leading coefficient is nonsingular.

\medskip
An open dense subset of $\mathbb{C}^k$ is usually called a generic set. 
Generic subsets of $\mathbb{C}^k$ that are described through matrices arise in applications 
from differential equations. A natural way to study such subspaces is by describing their eigenstructure. It is well known that elements of generic sets of regular matrix polynomials 
have only simple eigenvalues and therefore the corresponding block companion matrices are 
always diagonalizable. However, there are 
regular matrix polynomials whose block companion matrix is not diagonalizable. One such 
matrix polynomial is $P(\lambda) = I \lambda^2 + U_1 \lambda + U_0$ where, $U_0 = 
\begin{bmatrix*}[r]
     -\frac{1}{2} & -\frac{\sqrt{3}}{2} \\
      \frac{\sqrt{3}}{2} & -\frac{1}{2}
\end{bmatrix*}$	and 
$U_1 = \begin{bmatrix*}[r]
    	-1 & 0 \\
    	 0 & 1
    	\end{bmatrix*}$. More about this matrix polynomial will be discussed in Remarks 
\ref{rem-1}. There are also matrix polynomials with very {\it simple} matrix coefficients, 
having repeated eigenvalues but still the corresponding block companion matrix is 
diagonalizable. For instance, consider $P(\lambda) = I \lambda^2 + I \lambda + I$. We are therefore interested in subclasses of regular matrix polynomials whose block companion matrix is always diagonalizable. We prove in this 
manuscript an affirmative result for quadratic matrix polynomials whose coefficients are 
commuting unitaries or doubly stochastic matrices. This then lets us derive a 
Hoffman-Wielandt type inequality for the block companion matrices of such matrix polynomials. 
It turns out that one cannot go for higher degree matrix polynomials to get an affirmative result.

\medskip
Given a matrix polynomial $P(\lambda) = I \lambda^m + A_{m-1}\lambda^{m-1} + \cdots + 
A_1 \lambda +A_0$, let us consider the following matrix equation $X^m + A_{m-1}X^{m-1} 
+ \cdots + A_1X + A_0 = 0$ with $X \in M_n(\mathbb{C})$. An $n \times n$ matrix $X$ 
satisfying this equation is called a right solvent of the equation. We just 
call a right solvent $X$ a solution here. There are at most $\binom{nm}{n}$ solutions to 
this equation \cite{Fuchs-Schwarz}. Solutions $X_1, X_2, \ldots, X_m$ of the above matrix 
equation are said to be independent if the block Vandermonde matrix, 
$V:= \begin{bmatrix*}[c]
	I & I & \cdots & I \\
	X_1 & X_2 & \cdots & X_m \\
	\vdots & \vdots & \ddots & \vdots \\
	X_1^{m-1} & X_2^{m-1} & \cdots & X^{m-1}_m
	\end{bmatrix*}$ is invertible. In such a case, the corresponding block companion 
matrix $C$ of $P(\lambda)$ is similar to the  block diagonal matrix 
$\begin{bmatrix}
	X_1 & 0 & \cdots & 0\\
	0 & X_2 & \cdots & 0\\
	\vdots & \vdots & \ddots & \vdots \\
	0 & 0 & \cdots & X_m
	\end{bmatrix}$ through the block Vandermonde matrix $V$. Details of these may be 
found in \cite{Connes-Schwarz}. Thus if each of these solutions happen to be 
diagonalizable, then the block companion matrix $C$ is also diagonalizable. We shall 
use this later on in proving that the block companion matrix of a quadratic matrix 
polynomial with commuting unitary coefficients is diagonalizable. 

\medskip
The manuscript is organized as follows. Sections \ref{sec-1} and \ref{sec-2} are 
introductory and contain a brief introduction to matrix polynomials that are needed 
for this manuscript. The main results are presented in Section \ref{sec-3}, which is 
further divided into subsections for ease of reading. Each of these subsections is 
self-explanatory. The Hoffman-Wielandt type inequality is derived for the 
corresponding block companion matrices of matrix polynomials (with appropriate 
assumptions) in each of these subsections. In Section \ref{sec-4} the spectral 
condition number of a matrix $X$, which appears in the Hoffman-Wielandt type 
inequality  is estimated. The manuscript ends with some concluding remarks. Wherever 
possible, necessary remarks and examples are provided to justify the assumptions made. 
Most of the examples were computationally verified using Matlab and SageMath.

\section{Main Results}\label{sec-3}

This section contains the main results of this paper. As mentioned in the introduction, 
the inequality \eqref{HW-ineq-1} for block companion matrices of matrix polynomials 
whose coefficients are normal matrices fails to hold in general. For instance, consider 
$P(\lambda) = \begin{bmatrix*}[r]
				2 & 0 \\
				0 & -2
              \end{bmatrix*} \lambda +                                                                                                                                            \begin{bmatrix*}[r]
2 & 2 \\
2 & -14
\end{bmatrix*}$ and $Q(\lambda) = \begin{bmatrix*}[r]
									1 & 0 \\
									0 & -\frac{5}{4}
									\end{bmatrix*}  \lambda + \begin{bmatrix*}[r]
																2 & 5 \\
																5 & -\frac{30}{4}
																\end{bmatrix*}$  with 
the corresponding block companion matrices $C=\begin{bmatrix*}[r]
													-1 & -1 \\
													1 & -7
												\end{bmatrix*}$ and 
$D = \begin{bmatrix*}[r]
   	  -2 & -5 \\
	   4 & -6
	 \end{bmatrix*}$ respectively. Note that the eigenvalues of $C$ and $D$ are 
$\lambda_1 = - 4 - 2\sqrt{2}, \lambda_2 = - 4 + 2 \sqrt{2}$ 
and $\mu_1 = - 4 + 4i, \mu_2 = -4 - 4i$ respectively. One can verify that for any 
permutation $\pi$ on $\{1,2\}, \ 
\displaystyle \sum_{i=1}^{2}| \lambda_{i} - \mu_{\pi(i)} |^2 = 48 $, whereas 
$||C-D||_F^2 = 27$. Inequality \eqref{HW-ineq-1} also fails in general for 
quadratic matrix polynomials whose coefficients are either normal, unitary or 
doubly stochastic matrices. For the sake of brevity, we present below an example, 
where the coefficients are unitary. Let 
$P(\lambda) = 
\begin{bmatrix}
	1 & 0 \\
	0 & 1 
\end{bmatrix} \lambda^2 + 
\begin{bmatrix*}[r]
\frac{1}{\sqrt{2}} & \frac{1}{\sqrt{2}}\\
\frac{1}{\sqrt{2}} & -\frac{1}{\sqrt{2}}
\end{bmatrix*} \lambda + 
\begin{bmatrix*}[r]
\frac{4}{\sqrt{41}} & \frac{5}{\sqrt{41}}\\
\frac{5}{\sqrt{41}} & -\frac{4}{\sqrt{41}}
\end{bmatrix*}$ and 
$Q(\lambda) = \begin{bmatrix}
1 & 0 \\
0 & 1 \\
\end{bmatrix} \lambda^2 + \begin{bmatrix*}[r]
\frac{1}{\sqrt{2}} & \frac{1}{\sqrt{2}} \\
\frac{1}{\sqrt{2}} & -\frac{1}{\sqrt{2}} \\
\end{bmatrix*} \lambda + \begin{bmatrix*}[r]
-\frac{1}{2} & \frac{\sqrt{3}}{2} \\
-\frac{\sqrt{3}}{2} & -\frac{1}{2} 
\end{bmatrix*}$. If $C$ and $D$ are the corresponding block companion matrices of these 
matrix polynomials, then $||C - D||_F^2 = 4$. However, on computing the eigenvalues 
of $C$ and $D$, one observes that for any permutation $\pi$ on $\{1,2,3,4\}$, 
$\displaystyle \sum_{i=1}^{4} |\mu_{\pi(i)} - \lambda_i|^2 \geq 4.5102 > 4$, thereby 
proving that the inequality \eqref{HW-ineq-1} fails to hold. It therefore makes it pertinent 
to consider matrix polynomials whose block companion matrices satisfy the Hoffman-Wielandt 
type inequality. We set out to do this for quadratic matrix polynomials whose coefficients 
are commuting unitaries in Section \ref{sec-3.1} and later on in Section \ref{sec-3.2} for 
matrix polynomials whose coefficients are doubly stochastic matrices, with some 
additional assumptions. Having said these, one can easily verify that the 
inequality \eqref{HW-ineq-1} holds for block companion matrices of linear matrix 
polynomials whose coefficients are unitary or upper triangular. 
	
\subsection{Hoffman-Wielandt type inequality for block companion matrices of matrix 
polynomials with unitary coefficients} \hspace*{\fill} 
\label{sec-3.1}

\medskip

We begin with the following theorem. The proof is skipped.

\begin{theorem}\label{thm-1}
Let $P(\lambda) = A_1 \lambda + A_0$ be a linear matrix polynomial. Then the corresponding 
block companion matrix $C$ of $P(\lambda)$ is diagonalizable 
\begin{itemize}
\item[(1)] when the coefficients are unitary matrices. 
\item[(2)] when the coefficients are diagonal matrices.  
\item[(3)] when the coefficients are positive (semi)definite matrices.
\end{itemize}
\end{theorem}

\begin{remark}
\label{rem-0}

It is not hard to construct linear matrix polynomials whose coefficients 
are either normal or upper (lower) triangular such that the 
corresponding block companion matrix $C$ is not diagonalizable. If $P(\lambda)$ is a 
quadratic matrix polynomial whose coefficients are either 
(a) diagonal matrices (b) normal matrices (c) upper (lower) triangular matrices 
or (d) positive (semi)definite, then again, the corresponding block companion matrix need 
not be diagonalizable. The following matrix polynomial serves as an example for all of 
the above cases: Let $P(\lambda) = 
\begin{bmatrix*}[c]
      1 & 0\\
	  0 & 1
\end{bmatrix*} \lambda^2 + 
\begin{bmatrix*}[r]
	  2 & 0\\
	  0 & 2
\end{bmatrix*} \lambda + 
\begin{bmatrix*}[r]
	1 & 0\\
	0 & 2
\end{bmatrix*}$.
\end{remark}

The above remark also implies that if a matrix polynomial has diagonal coefficients, it 
does not necessarily mean that the corresponding block companion matrix is diagonalizable. 
Theorem \ref{thm-1} and the remark that follows suggest that matrix polynomials with 
unitary coefficients might form a good candidate as far as diagonalizability of the block 
companion matrix is concerned. We discuss this below. We state the following interesting 
result due to Cameron (see Theorems $3.2$ and $3.3$, \cite{Cameron}).

\begin{theorem}\label{eigenvalue-unitary-coefficients}
\
\begin{enumerate}
\item If $P(\lambda)$ is a matrix polynomial with unitary coefficients, then the 
eigenvalues of $P(\lambda)$ lie in the annular region $\frac{1}{2} < |\lambda| < 2$.
\item Let \ $\mathcal{U} = \Big \{P(\lambda) = \displaystyle \sum_{i=0}^{m}U_i\lambda^i: 
U_i$'s are $n\times n$ unitary matrices and $m,n \in \mathbb{N}\Big \}$ and let 
$S_{\mathcal{U}} = \{|\lambda_0|: \lambda_0$ is an eigenvalue of 
$P(\lambda)\in \mathcal{U}\}$. Then, $\inf S_{\mathcal{U}} = \frac{1}{2}$ and 
$\sup S_{\mathcal{U}}=2$.
\end{enumerate}
\end{theorem}

In a later section on matrix polynomials with doubly stochastic coefficients, we will 
give bounds on eigenvalues similar to that of Theorem \ref{eigenvalue-unitary-coefficients}. 
Let us observe that it suffices to consider monic matrix polynomials while dealing with 
commuting unitary coefficients. For, if $P(\lambda) = V_2 \lambda^2 + V_1 \lambda + V_0$, 
where the $V_i$'s are commuting unitary matrices, then we can consider the corresponding 
monic matrix polynomial $P_U(\lambda) = I \lambda^2 + U_1 \lambda + U_0$ and observe that 
the coefficients of $P_U$ are also commuting unitary matrices. With these observations, 
we have the following theorem.
	
\begin{theorem}\label{unitary-thm-1}
Let $P(\lambda)= I \lambda^2 + U_1 \lambda + U_0$ be an $n\times n$ matrix polynomial 
where the coefficients $U_0$ and $U_1$ are commuting unitary matrices. Then the 
corresponding block companion matrix $C$ of $P(\lambda)$ is diagonalizable.
\end{theorem}
	
\begin{proof}
The matrices $U_0$ and $U_1$ being commuting unitary matrices, there exists an 
$n\times n$ unitary matrix $W$ such that $WU_1W^*=D_1$ and $WU_0W^*=D_0$ where, 
$D_1,D_0$ are diagonal matrices whose diagonal entries are the eigenvalues of 
$U_1$ and $U_0$ respectively. Let $D_1 = \text{diag}(a_{11},a_{22},\ldots,a_{nn})$ 
and $D_0 = \text{diag}(b_{11},b_{22},\ldots,b_{nn})$. Let 
$Q(\lambda):= WP(\lambda)W^*  = I \lambda^2 + D_1 \lambda + D_0$. The corresponding 
block companion matrix of $Q(\lambda)$ is 
$D = \begin{bmatrix*}[c]
	0 & I \\
   -D_0 & -D_1
	\end{bmatrix*}$. Note that $C$ and $D$ are similar through the unitary matrix 
$U:= W \oplus W$. It therefore suffices to prove that the matrix $D$ is diagonalizable. 
Consider
$Q(\lambda) = I \lambda^2 + D_1 \lambda + D_0
		= \begin{bmatrix*}[c]
		  \lambda^2 + a_{11} \lambda+b_{11} & 0 & \cdots & 0\\
			0 & \lambda^2 + a_{22}\lambda + b_{22} & \cdots & 0\\
			\vdots & \vdots & \ddots & \vdots \\
			0 & 0 & \cdots & \lambda^2+a_{nn}\lambda + b_{nn}
		\end{bmatrix*}$.
Let $f_{ii}(\lambda):= \lambda^2 + a_{ii}\lambda + b_{ii}, \ 1 \leq i \leq n$. 
Thus  
$Q(\lambda)=\begin{bmatrix*}[c]
			f_{11}(\lambda) & 0 & \cdots & 0\\
			0 & f_{22}(\lambda) & \cdots & 0\\
			\vdots & \vdots & \ddots & \vdots \\
			0 & 0 & \cdots & f_{nn}(\lambda)
		\end{bmatrix*}$. Observe that for each $i, \ 1 \leq i \leq n$, the 
polynomial $f_{ii}(\lambda)$ has two distinct roots. For otherwise, 
$f_{ii}(\lambda)=(\lambda-\lambda_0)^2$ for some $\lambda_0$, so that 
$\lambda^2 + a_{ii}\lambda + b_{ii} = \lambda^2 - 2 \lambda_0 \lambda + \lambda_0^2$. 
Comparing the coefficients we get, $a_{ii} = -2\lambda_0$ and $b_{ii} = \lambda_0^2$. 
Taking the modulus we have, $|\lambda_0| = \frac{1}{2}$ and $|\lambda_0| = 1$, which 
is not possible. An alternate argument follows from the first statement of 
Theorem \ref{eigenvalue-unitary-coefficients}. This proves the assertion that 
$f_{ii}(\lambda)$ has two distinct roots. Let $\lambda_i \neq \mu_i$ be two distinct 
roots of $f_{ii}(\lambda)$ for $1 \leq i \leq n$. Let us now define 
$X_1:= \text{diag}(\lambda_1,\lambda_2,\ldots,\lambda_n)$
and $X_2:= \text{diag}(\mu_1,\mu_2,\ldots,\mu_n)$. It is easy to verify that $X_1$ 
and $X_2$ satisfy the quadratic matrix equation $X^2 + D_1X + D_0 = 0$. Moreover, 
$\text{det}(X_1 - X_2) = \displaystyle \prod_{i=1}^{n}(\lambda_i-\mu_i) \neq 0$, 
as $\lambda_i \neq \mu_i$ for all $i$. Therefore the block Vandermonde 
matrix, as defined in Section \ref{sec-2}, is invertible; in other words, $X_1$ and 
$X_2$ are two independent solutions to the matrix equation $X^2 + D_1X + D_0 = 0$. 
Once again, it follows from what has been mentioned earlier in Section \ref{sec-2}, 
that the block companion matrix $D=\begin{bmatrix*}
		      	    0 & I \\
			       -D_0 & -D_1
		          \end{bmatrix*}$ is similar to the block diagonal matrix 
$\widetilde{D} = \begin{bmatrix*}
 X_1 & 0 \\
 0 & X_2
\end{bmatrix*}$ through the block Vandermonde matrix $ V = \begin{bmatrix*}
	I & I \\
	X_1 & X_2
\end{bmatrix*}$. Since both $X_1$ and $X_2$ are diagonal, $\widetilde{D}$ is diagonal. 
Hence $D$ is diagonalizable.	
\end{proof} 

Few remarks are in order.

\begin{remark}
\label{rem-1}
\
\begin{enumerate}
	
\item Theorem \ref{unitary-thm-1} not only proves diagonalizability of 
the block companion matrix $D$, but also explicitly gives the invertible block 
Vandermonde matrix $V$ through which diagonalization happens. Note that the block 
companion matrix $C$ is then diagonalizable through the matrix $X =UV$, where $U$ is 
the unitary matrix from the above theorem. In Section \ref{sec-4} we prove that 
$\kappa(X) < 2$.	
	
\item Theorem \ref{unitary-thm-1} does not hold good if the unitary matrices $U_0$ and 
$U_1$ do not commute. For example, when $n=2$ consider the matrix polynomial 
$P(\lambda) = I \lambda^2 + U_1 \lambda + U_0$ where, $U_0 = 
\begin{bmatrix*}[r]
     -\frac{1}{2} & -\frac{\sqrt{3}}{2} \\
      \frac{\sqrt{3}}{2} & -\frac{1}{2}
\end{bmatrix*}$	and 
$U_1 = \begin{bmatrix*}[r]
    	-1 & 0 \\
    	 0 & 1
    	\end{bmatrix*}$. Note that $U_1$ and $U_0$ are unitary matrices that do not 
commute. One can easily check that the corresponding block companion matrix $C$ of 
$P(\lambda)$ has two distinct eigenvalues $1$ and $-1$, each of algebraic multiplicity 
$2$. However, a simple computation reveals that the geometric multiplicity of both 
these eigenvalues is $1$. Hence $C$ is not diagonalizable. 

\medskip
\noindent		
For $n = 3$, consider $P(\lambda)=\begin{bmatrix*}[c]
    			1 & 0 & 0 \\
    			0 & 1 & 0 \\
    			0 & 0 & 1
    		\end{bmatrix*}\lambda^2 + \begin{bmatrix*}[c]
    			1 & 0 & 0 \\
    			0 & 0 & 1 \\
    			0 & 1 & 0
    		\end{bmatrix*}\lambda + \begin{bmatrix*}[c]
    			0 & 1 & 0 \\
    			0 & 0 & 1 \\
    			1 & 0 & 0
    		\end{bmatrix*}$ with non-commuting coefficients. The corresponding 
block companion matrix $C$ of $P(\lambda)$ is not diagonalizable in this case too. Note 
that one can extend this example and see that the block companion matrix $C$ 
is not diagonalizable for any size $n \geq 4$ as well.
    		
\item If the degree of matrix polynomial is greater than $2$, the corresponding block 
companion matrix need not be diagonalizable, even if all coefficients are commuting 
unitary matrices. For example consider 
$P(\lambda) = \begin{bmatrix*}[r]
    			1 & 0 \\
    			0 & 1
    		\end{bmatrix*}\lambda^3 + \begin{bmatrix*}[r]
    									-1 & 0 \\
    									 0 & -1
    								   \end{bmatrix*}\lambda^2 + 
    		\begin{bmatrix*}[r]
    			-1 & 0\\
    			0 & -1
    		\end{bmatrix*}\lambda + \begin{bmatrix*}[r]
    			1 & 0\\
    			0 & 1
    		\end{bmatrix*}$. The corresponding block companion matrix $C$ of $P(\lambda)$ 
is not diagonalizable. The argument is the same as in the previous remark. 	     
    		
\end{enumerate}
\end{remark}

The above remark justifies the need to consider only quadratic matrix polynomials with 
commuting unitary coefficients. We now deduce the Hoffman-Wielandt type inequality for
quadratic matrix polynomials with unitary coefficients.
     
\begin{theorem}\label{HW-thm-1}
Let $P$ and $Q$ be quadratic matrix polynomials of same size, where $P$ 
satisfies the conditions of Theorem \ref{unitary-thm-1}. If $C$ and $D$ are the 
corresponding block companion matrices, then there exists a permutation $\pi$ of 
the indices $1, 2, \ldots, 2n$ such that 
$\displaystyle \sum_{i=1}^{2n} |\alpha_i - \beta_{\pi(i)}|^2 
\leq ||X||^2_2 ||X^{-1}||^2_2 ||C-D||^2_F$, where $\{\alpha_i\}$ and $\{\beta_i\}$ are 
the eigenvalues of $C$ and $D$ respectively, and $X$ is a nonsingular matrix whose 
columns are the eigenvectors of $C$.
\end{theorem}
     
\begin{proof}
The assumptions on $P$ ensure that the matrix $C$ is diagonalizable. The result now 
follows from Theorem \ref{starting theorem}.
\end{proof}
 
The Hoffman-Wielandt type inequality for linear matrix polynomials is stated below. 
We skip the proof as it is similar to above theorem.
    
\begin{theorem}\label{HW-thm-1.1}
Let $P$ and $Q$ be linear matrix polynomials of same size, where $P$ satisfies any of the 
conditions of Theorem \ref{thm-1}. If $C$ and $D$ are 
the corresponding block companion matrices, then there exists a permutation $\pi$ of 
the indices $1, 2, \ldots, n$ such that 
$\displaystyle \sum_{i=1}^{n} |\alpha_i - \beta_{\pi(i)}|^2 
\leq ||X||^2_2 ||X^{-1}||^2_2 ||C-D||^2_F$, where $\{\alpha_i\}$ and $\{\beta_i\}$ 
are the eigenvalues of $C$ and $D$ respectively, and $X$ is a nonsingular 
matrix whose columns are the eigenvectors of $C$.
\end{theorem}

\subsection{Hoffman-Wielandt type inequality for block companion matrices of matrix 
polynomials with doubly stochastic coefficients}
\hspace*{\fill}
\label{sec-3.2}   

\medskip

We now consider matrix polynomials whose coefficients are doubly stochastic matrices. 
Recall that a nonnegative square matrix is doubly stochastic if all the row and column sums are 
$1$. A classical result of Birkhoff says that any doubly stochastic matrix is a 
convex combination of permutation matrices (see Section $8.7$ of \cite{Horn-Johnson}). 
Since permutation matrices are unitary, it is natural to ask if the Hoffman-Wielandt type 
inequality holds for matrix polynomials with doubly stochastic coefficients. We explore 
this question in this section.

We begin this section with eigenvalue bounds for matrix polynomials whose coefficients 
are doubly stochastic matrices, with some added assumptions. We state this below and 
skip the proof as the proof technique is essentially the same as 
in \cite{Cameron} with the observation that the spectral norm of a doubly stochastic matrix 
is $1$ and the inverse of a permutation matrix is again a permutation matrix. We illustrate 
via examples the validity of these assumptions. We also state a result analogous to the 
second statement of Theorem \ref{eigenvalue-unitary-coefficients}. Once again, we skip 
the proof to avoid repetitiveness. 
    
\begin{theorem}\label{doubly-stochastic-thm-1-2}
\
\begin{enumerate}
\item Let $P(\lambda) = A_m \lambda^m + A_{m-1} \lambda^{m-1} + \cdots + A_1 \lambda + A_0$, 
where $A_m$, $A_0$ are $n \times n$ permutation matrices and $A_{m-1}, \ldots, A_1$ are 
$n \times n$ doubly stochastic matrices. If $\lambda_0$ is an eigenvalue of $P(\lambda)$, 
then $\frac{1}{2} < |\lambda_0| < 2$.
\item Let $\mathcal{D} = \Big\{P(\lambda) = \displaystyle \sum_{i=0}^{m}A_i\lambda^i: A_i$'s 
are $n\times n$ doubly stochastic matrices, $A_m, A_0$ are permutation matrices and 
$m, n \in \mathbb{N}$ $\Big \}$ and let $S_{\mathcal{D}} = \{|\lambda_0|: \lambda_0$ is an 
eigenvalue of $P(\lambda) \in \mathcal{D}$ $\}$. Then $\inf S_{\mathcal{D}} = \frac{1}{2}$ 
and $\sup S_{\mathcal{D}} = 2$.
\end{enumerate}
\end{theorem}
    
Few remarks are in order.
  
\begin{remark}\label{rem-2}

If the leading coefficient or the constant term (or both) is a doubly stochastic matrix, 
but not a permutation matrix, then the eigenvalues may not necessarily lie in the 
region $\frac{1}{2} < |\lambda| < 2$. The following examples illustrate this. 
Let $P(\lambda) = I \lambda^2 + \begin{bmatrix*}[c]
    			\frac{1}{4} & \frac{3}{4} \\	
    			\frac{3}{4} & \frac{1}{4}
    		\end{bmatrix*} \lambda + \begin{bmatrix*}[c]
    			\frac{1}{3} & \frac{2}{3}\\
    			\frac{2}{3} & \frac{1}{3}
    		\end{bmatrix*}$. One of the eigenvalue is $\frac{3- \sqrt{57}}{12} 
= -0.3792 $, which is less than $\frac{1}{2}$ in absolute value. 
Similarly if $P(\lambda) = \begin{bmatrix*}[c]
    			\frac{1}{3} & \frac{2}{3}\\
    			\frac{2}{3} & \frac{1}{3}
    		\end{bmatrix*}\lambda^2 + \begin{bmatrix*}[c]
    			\frac{1}{4} & \frac{3}{4} \\	
    			\frac{3}{4} & \frac{1}{4}
    		\end{bmatrix*} \lambda +\begin{bmatrix*}[c]
    			1 & 0 \\
    			0 & 1
    		\end{bmatrix*}$. The eigenvalues of $P(\lambda)$ are 
$\frac{-1 \pm i\sqrt{3}}{2}$ and $\frac{-3 \pm \sqrt{57}}{4}$ and the absolute value of 
$\frac{-3-\sqrt{57}}{4}$ is $2.637 > 2$. 
\end{remark}
       
We are now in a position to derive diagonalizability of the block companion matrix of 
a matrix polynomial with doubly stochastic coefficients. We begin with some easy observations.
    
\begin{theorem}\label{doubly-stochastic-thm-4}
Let $P(\lambda) = A_1 \lambda + A_0$ be a linear matrix polynomial whose coefficients are 
$2 \times 2$ doubly stochastic matrices. Then, the corresponding block companion matrix is 
diagonalizable. 
\end{theorem}

\begin{proof}
Writing $A_0$ and $A_1$ as 
$A_0 =\begin{bmatrix*}[c]
    		b & 1-b \\
    		1-b & b
     \end{bmatrix*}, \ A_1 = \begin{bmatrix*}[c]
    		a & 1-a \\
    		1-a & a
    	\end{bmatrix*}$ where, $0 \leq a, b \leq 1$, we observe that  
$C = -\begin{bmatrix*}[c]
    		\frac{a+b-1}{2a-1} & \frac{a-b}{2a-1} \\
    		\frac{a-b}{2a-1} & \frac{a+b-1}{2a-1}
    	\end{bmatrix*}$, a real symmetric matrix. Hence $C$ is diagonalizable.
\end{proof}

\begin{remark}\label{rem-3}

Let $P(\lambda)=A_1\lambda+A_0$ be an $n\times n$ matrix polynomial with $n\geq 3$. If one 
of $A_1$ or $A_0$ is a doubly stochastic matrix which is not a permutation matrix, then 
$C$ need not be diagonalizable. For example consider $P(\lambda)=\begin{bmatrix*}[c]
       		1 & 0 & 0 \\
       		0 & 1 & 0 \\
       		0 & 0 & 1
       	\end{bmatrix*}\lambda+\begin{bmatrix*}[c]
       	\frac{1}{8} & \frac{1}{2} & \frac{3}{8} \\
       	\frac{1}{4} & \frac{3}{8} & \frac{3}{8} \\
       	\frac{5}{8} & \frac{1}{8} & \frac{1}{4}
       \end{bmatrix*}$. We can check that $C$ is not diagonalizable. 
\end{remark}

\medskip
\noindent
Let us now prove that the block companion matrix of a quadratic matrix 
polynomial with $2 \times 2$ doubly stochastic coefficients is diagonalizable. We make use 
of Theorem \ref{doubly-stochastic-thm-1-2} to prove this.
    
\begin{theorem}\label{doubly-stochastic-thm-5}
Let $P(\lambda) = A_2\lambda^2 + A_1 \lambda + A_0$ where, $A_2, A_0$ are $2 \times 2$ 
permutation matrices and $A_1$ is a $2 \times 2$ doubly stochastic matrix. Then the 
corresponding block companion matrix $C$ of $P(\lambda)$ is diagonalizable.
\end{theorem}

\begin{proof}
The proof involves three cases.

\medskip
\noindent
Case 1: Suppose $A_2 = A_1 = A_0 = I$.
Then $P(\lambda)$ can be written as, $P(\lambda) = \begin{bmatrix*}[c]
    		\lambda^2 + \lambda + 1 & 0 \\
    		0 & \lambda^2 + \lambda + 1
    	\end{bmatrix*}$. Therefore $\frac{-1 + i\sqrt{3}}{2}$ and 
$\frac{-1 - i\sqrt{3}}{2}$ 
are eigenvalues of $P(\lambda)$ and hence of $C$, of multiplicity $2$ each.
We can check that $[1,0]^t$ and $[0,1]^t$ are eigenvectors of $P(\lambda)$ 
corresponding to both the eigenvalues $\frac{-1 + i\sqrt{3}}{2}$
and $\frac{-1 - i\sqrt{3}}{2}$. 
Therefore $u_1 = [1, 0, \frac{-1 + i\sqrt{3}}{2}, 0]^t, 
\ u_2 = [0, 1, 0, \frac{-1 + i\sqrt{3}}{2}]^t, \ 
u_3 = [1, 0, \frac{-1 - i\sqrt{3}}{2}, 0]^t$ and 
$u_4 = [0, 1, 0, \frac{-1 - i\sqrt{3}}{2}]^t$ are linearly independent eigenvectors 
of $C$. This proves diagonalizability of $C$.
    	
\medskip
\noindent
Case 2: If $A_2 = A_1 = A_0 = I^{\prime}$, where $I^{\prime} = 
\begin{bmatrix}
		0 & 1 \\
   		1 & 0
\end{bmatrix}$. Then the monic matrix polynomial corresponding to
$P(\lambda)$ is $P_U(\lambda) = I \lambda^2 + I \lambda + I$. Hence by Case 1, 
$C$ is diagonalizable.
    	
\medskip
\noindent
Case 3: Consider the corresponding monic matrix polynomial, $P_U(\lambda) = I \lambda^2 
+ B_1 \lambda + B_0$ where, $B_1 = A_2^{-1}A_1$ is a doubly stochastic matrix and 
$B_0 = A_2^{-1}A_0$ is permutation matrix. Let $B_1 = 
\begin{bmatrix*}[c]
      a & 1-a \\
     1-a & a
\end{bmatrix*}$ and $B_0 = \begin{bmatrix*}[c]
    						b & 1-b \\
    						1-b & b
    					   \end{bmatrix*}$ where, $0 \leq a,b \leq 1$. Then
$P_U(\lambda) = \begin{bmatrix*}[c]
    		\lambda^2 + a \lambda + b & (1-a)\lambda + (1-b)\\
    		(1-a)\lambda+(1-b) & \lambda^2+a\lambda+b
    	\end{bmatrix*}$ and $\text{det}P_U(\lambda) = (\lambda^2 + \lambda+1)
(\lambda^2 + (2a-1) \lambda + (2b-1))$. Note that $\lambda^2 + \lambda + 1 
\neq \lambda^2 + (2a-1) \lambda + (2b-1)$. Otherwise $a = b = 1$ which then will 
imply that $A_2 = A_1 = A_0 = I$ or $A_2 = A_1 = A_0 = I^{\prime}$. Moreover, 
since both $\lambda^2 + \lambda + 1$ and $\lambda^2 + (2a-1)\lambda + (2b-1)$ are 
real polynomials they do not have common roots. Now we claim that 
$\lambda^2 + (2a-1)\lambda + (2b-1)$ has two distinct roots. Suppose there is only 
one root, say, $\lambda_0$. Then we have $\lambda^2 + (2a-1)\lambda + (2b-1) = 
(\lambda - \lambda_0)^2 = \lambda^2 - 2\lambda_0 \lambda + \lambda_0^2$. Comparing 
the coefficients, we get $2a-1 = -2\lambda_0$. Since $0\leq a\leq 1$, we have 
$-1 \leq 2a-1 \leq 1$. This implies $2|\lambda_0| = |2a-1| \leq 1$. 
Therefore $|\lambda_0| \leq \frac{1}{2}$, a contradiction to 
Theorem \ref{doubly-stochastic-thm-1-2}. Thus $\lambda^2 + (2a-1)\lambda + (2b-1)$ 
has two distinct roots. Since $\lambda^2 + \lambda + 1$ also has two distinct roots, 
$P(\lambda)$ and hence $C$ has four distinct eigenvalues. Hence $C$ is diagonalizable.
\end{proof}

The following remarks justify the assumptions made in the above theorem.

\begin{remark}\label{rem-4}
\
\begin{enumerate}
\item In the above theorem, if the leading coefficient or the constant term (or both) is a doubly stochastic 
matrix, but not a permutation matrix, then the corresponding block companion matrix need not be 
diagonalizable. For example, consider $P(\lambda) = \begin{bmatrix}
													\frac{11}{24} & \frac{13}{24}\\
													\frac{13}{24} & \frac{11}{24}
													\end{bmatrix} \lambda^2 + 
\begin{bmatrix}
\frac{1}{4} & \frac{3}{4}\\
\frac{3}{4} & \frac{1}{4}
\end{bmatrix} \lambda + \begin{bmatrix}
											\frac{1}{8} & \frac{7}{8}\\
											\frac{7}{8} & \frac{1}{8}
											\end{bmatrix}$. We can check that the 
corresponding block companion matrix is not diagonalizable. One can also look at the matrix polynomial 
$P(\lambda) = \begin{bmatrix}
						  1 & 0\\
						  0 & 1
					   \end{bmatrix} \lambda^2 + \begin{bmatrix}
					   													 \frac{1}{2} & \frac{1}{2}\\
					   													 \frac{1}{2} & \frac{1}{2}  
					   											       \end{bmatrix} \lambda + 
\begin{bmatrix}
\frac{1}{2} & \frac{1}{2}\\
\frac{1}{2} & \frac{1}{2}
\end{bmatrix}$. 

\item If the size of a quadratic matrix polynomial $P(\lambda)$ is greater than 
$2$ then the corresponding block companion matrix $C$ need not be diagonalizable even 
when all the coefficients are permutation matrices (see the Example in 
Remark \ref{rem-1}). Note that the coefficients of the matrix polynomial in that example 
are non-commuting permutation matrices. However when the coefficients of $P(\lambda)$ 
are commuting permutation matrices the corresponding block companion matrix $C$ is 
diagonalizable, as already proved in Theorem \ref{unitary-thm-1}.

\item Let $P(\lambda) = A_2 \lambda^2 + A_1\lambda + A_0$ be an $n \times n$ 
matrix polynomial with $n \geq 3$. If one of $A_2, A_1, A_0$ is a doubly stochastic 
matrix which is not permutation, then the corresponding block companion matrix $C$ 
need not be diagonalizable even if the coefficients commute. For example consider \\
$P(\lambda) = \begin{bmatrix*}[c]
    			1 & 0 & 0 \\
    			0 & 1 & 0 \\
    			0 & 0 & 1
    		\end{bmatrix*}\lambda^2 + \begin{bmatrix*}[c]
    			\frac{5}{12} & \frac{5}{12} & \frac{1}{6} \\
    			\frac{1}{4} & \frac{1}{4} & \frac{1}{2} \\
    			\frac{1}{3} & \frac{1}{3} & \frac{1}{3}
    		\end{bmatrix*}\lambda + 
    		\begin{bmatrix*}[c]
    			1 & 0 & 0 \\
    			0 & 1 & 0 \\
    			0 & 0 & 1
    		\end{bmatrix*}$. We can check that the corresponding block companion matrix 
of $P(\lambda)$ is not diagonalizable. 
\item Let $P(\lambda)$ be a matrix polynomial of degree greater than $2$. Then the 
corresponding block companion matrix need not be diagonalizable even with coefficients 
of $P(\lambda)$ being commuting permutation matrices.
For example consider $P(\lambda) = 
    		\begin{bmatrix*}[c]
    			1 & 0 \\
    			0 & 1
    		\end{bmatrix*}\lambda^3 + \begin{bmatrix*}[c]
    			0 & 1 \\
    			1 & 0
    		\end{bmatrix*}\lambda^2 + 
    		\begin{bmatrix*}[c]
    			0 & 1\\
    			1 & 0
    		\end{bmatrix*}\lambda + 
    		\begin{bmatrix*}[c]
    			1 & 0 \\
    			0 & 1
    		\end{bmatrix*}$. The coefficients are commuting permutation matrices. 
However, the corresponding block companion matrix is non-diagonalizable.
\end{enumerate}
\end{remark}

We end this section by pointing out that the Hoffman-Wielandt type inequality for matrix 
polynomials with doubly stochastic coefficients can be derived as in the unitary case. 
For the sake of completeness, we state below only the quadratic polynomials version 
and skip the proof. 
    
\begin{theorem}\label{HW-thm-2}
Let $P$ and $Q$ be quadratic matrix polynomials of same size, where $P$ satisfies 
conditions of Theorem \ref{doubly-stochastic-thm-5}. If $C$ and $D$ are 
the corresponding block companion matrices, then there exists a permutation $\pi$ of the 
indices $1,\ldots, 4$ such that 
$\displaystyle \sum_{i=1}^{4} |\alpha_i - \beta_{\pi(i)}|^2 
\leq ||X||^2_2 ||X^{-1}||^2_2 ||C-D||^2_F$, where $\{\alpha_i\}$ and 
$\{\beta_i\}$ are the eigenvalues of $C$ and $D$ respectively, and $X$ is a nonsingular 
matrix whose columns are the eigenvectors of $C$.
\end{theorem}

\subsection{Estimation of the spectral condition number}\label{sec-4}\hspace*{\fill}

\medskip

Computing/estimating the condition number of an invertible matrix can be quite hard. 
Notice that a diagonalizable matrix can be diagonalized through
more than one matrix and the spectral condition number of these matrices need not be 
the same. We also know from Theorems \ref{HW-thm-1} and \ref{HW-thm-2} that diagonalization 
of the block companion matrix of certain matrix polynomials is achieved through a particular 
block Vandermonde matrix $V$. This gives us some hope to estimate the spectral condition 
in the above mentioned theorems. We set out to do this in this section. 

\medskip
\subsubsection{Condition number of matrix $X$ that appears in Theorem \ref{HW-thm-1}}\hspace*{\fill}

\medskip

We first prove that the spectral condition number of the block Vandermonde matrix $V$ 
obtained in Theorem \ref{unitary-thm-1} is less than $2$. 

\begin{theorem}\label{thm-spectral condition number-1}
Let $V$ be the block Vandermonde matrix obtained in Theorem \ref{unitary-thm-1}. 
Then $\kappa(V)< 2$.
\end{theorem}

\begin{proof}
From Theorem \ref{unitary-thm-1} we have 
$V = \begin{bmatrix}
I & I \\
X_1 & X_2
\end{bmatrix}$ with $X_1 = \text{diag}(\lambda_1,\dots,\lambda_n)$ and 
$X_2=\text{diag}(\mu_1\dots,\mu_n)$, where $\lambda_i$ and $\mu_i$ are the distinct 
roots of $f_{ii}(\lambda) = \lambda^2 + a_{ii}\lambda + b_{ii}$ for $i = 1, \dots, n$ 
and $a_{ii}, b_{ii}$ are the eigenvalues of $U_1$, $U_0$ respectively. We thus have
$|\lambda_i + \mu_i| = 1$ and $|\lambda_i\mu_i| = 1$ for $i = 1, \dots, n$. It is then 
easy to show that $(\lambda_i-\mu_i)^2 = a_{ii}^2-4b_{ii}$. Thus, 
$|\lambda_i-\mu_i|^2=|a_{ii}^2-4b_{ii}|\geq ||a_{ii}^2|-4|b_{ii}||=|1-4|=3$. Using the 
parallelogram identity we have, 
$2(|\lambda_i|^2 + |\mu_i|^2) = |\lambda_i + \mu_i|^2 + |\lambda_i - \mu_i|^2 \geq 1+3 = 4$.
This implies that $|\lambda_i|^2 + |\mu_i|^2 \geq 2 $. 

Since $\big| |\lambda_i|-|\mu_i| \big| \leq |\lambda_i+\mu_i| = 1$, we have
$1 \geq (|\lambda_i| - |\mu_i|)^2 = |\lambda_i|^2 + |\mu_i|^2 - 2|\lambda_i| |\mu_i|$. 
Therefore $|\lambda_i|^2 + |\mu_i|^2 \leq 3$. We thus have proved the following:\\
\begin{equation}\label{eqn-norm}
2 \leq |\lambda_i|^2 + |\mu_i|^2 \leq 3, \hspace{0.5cm} i = 1, \ldots, n. 
\end{equation}

We know that $\kappa(V) = \displaystyle \frac{\sigma_{\text{max}}}{\sigma_{\text{min}}}$, 
where $\sigma_{\text{min}}$ and $\sigma_{\text{max}}$ are the smallest and the largest 
singular values of $V$ respectively. Since the singular values of $V$ are the positive 
square roots of the eigenvalues of $VV^*$, we estimate bounds for these eigenvalues.  

\medskip
Define $L:= \begin{bmatrix}
I & 0 \\
-(X_1+X_2)(2I- I \lambda)^{-1} & I
\end{bmatrix}$. Notice that $L$ is a matrix with $\text{det} = 1$. Let us now compute the matrix 
$L(VV^*-I \lambda) = \\
\begin{bmatrix}
2I -I\lambda  & X_1^*+X_2^* \\
0 & -(X_1+X_2)(2I-I\lambda )^{-1}(X_1^*+X_2^*) + X_1X_1^*+X_2X_2^*-I\lambda 
\end{bmatrix}$. Since $X_1$ and $X_2$ are diagonal matrices,
$\text{det}(VV^*-I\lambda ) 
= \text{det}\big(L(VV^*-I\lambda )\big) 
= \text{det}(2I-I\lambda )\big(-(X_1+X_2)(2I-I\lambda )^{-1}(X_1^*+X_2^*)+
(X_1X_1^*+X_2X_2^*-I\lambda )\big)
= \text{det}\big(I \lambda^2 - (2+X_1X_1^*+X_2X_2^*) \lambda+(X_1X_1^*+X_2X_2^*-X_1X_2^*
-X_2X_1^*)\big) = \prod_{i=1}^{n} \big(\lambda^2-(2+|\lambda_i|^2+|\mu_i|^2)\lambda +
(|\lambda_i|^2 + |\mu_i|^2 -\lambda_i\bar{\mu}_i-\mu_i\bar{\lambda}_i)\big)$.  
Thus, in order to compute the eigenvalues of $VV^*$, it suffices to determine the roots 
of the polynomials 
$\lambda^2-(2+|\lambda_i|^2+|\mu_i|^2)\lambda+(|\lambda_i|^2+|\mu_i|^2-\lambda_i
\bar{\mu}_i-\mu_i\bar{\lambda}_i)$. These are given by \\
$\alpha_i=\displaystyle \frac{(2+|\lambda_i|^2+|\mu_i|^2) - \sqrt{(2-(|\lambda_i|^2+
|\mu_i|^2))^2+4}}{2}$ and \\
$\beta_i=\displaystyle \frac{(2+|\lambda_i|^2+|\mu_i|^2) + \sqrt{(2-(|\lambda_i|^2+
|\mu_i|^2))^2+4}}{2}$, for $i = 1, \dots, n$.

Note that $\alpha_i \leq \beta_i$ as the term inside the square root symbol is positive. 
We now claim that $1 \leq \alpha_i \leq \beta_i < 4$ for all $i = 1, \dots, n$. 
Suppose on contrary, $\alpha_i < 1$ for some $i = 1, \dots, n$. Then we have,
\begin{align*}
& \displaystyle \frac{(2+|\lambda_i|^2+|\mu_i|^2)-\sqrt{(2-(|\lambda_i|^2+|\mu_i|^2))^2
+4}}{2} < 1 \\
\implies & (2+|\lambda_i|^2+|\mu_i|^2)-\sqrt{(2-(|\lambda_i|^2+|\mu_i|^2))^2+4}  < 2 \\
\implies & (|\lambda_i|^2+|\mu_i|^2)^2 < (2-(|\lambda_i|^2+|\mu_i|^2))^2+4 \\
\implies & (|\lambda_i|^2+|\mu_i|^2)^2 < 4 + (|\lambda_i|^2+|\mu_i|^2)^2 - 4(|\lambda_i|^2
+|\mu_i|^2)+4\\
\implies &  |\lambda_i|^2+|\mu_i|^2  < 2,
\end{align*}
a contradiction to the inequality \eqref{eqn-norm}. Therefore we have 
$1 \leq \alpha_i$ for all $i = 1, \dots, n$. Similarly, if $4 \leq \beta_i$ for some 
$i = 1, \dots, n$, then we have $3 < |\lambda_i|^2 + |\mu_i|^2$ which is again a contradiction 
to the inequality \eqref{eqn-norm}. Therefore $\beta_i < 4$ for all $i = 1, \dots, n$. We 
thus have $1 \leq \alpha_i \leq \beta_i < 4$, which implies that 
$1 \leq \sigma_{\text{min}}$ and $\sigma_{\text{max}} < 2$. Therefore
$\kappa(V) = \displaystyle \frac{\sigma_{\text{max}}}{\sigma_{\text{min}}} < 2$.
\end{proof}

As mentioned in the Remark \ref{rem-1}, the matrix which diagonalizes the block companion 
matrix $C$ is $X = UV$, where $U$ is a unitary matrix. Since the spectral condition number 
is unitarily invariant, we have $\kappa(X) = \kappa(V) < 2$. Thus in Theorem \ref{HW-thm-1} 
we have $||X||^2_2||X^{-1}||^2_2 < 4$.

\medskip
\subsubsection{Condition number of matrix $X$ obtained in Theorem \ref{HW-thm-1.1}}
\hspace*{\fill}

\medskip
 
In parts $(1)$ and $(2)$ of the Theorem \ref{thm-1} the block companion matrices  
are unitary and diagonal matrices respectively. Hence both are diagonalizable through 
unitary matrices, whose spectral condition number is $1$.

In part (3) of Theorem \ref{thm-1}, the block companion matrix is
$C = -A_1^{-1}A_0 = A_1^{-1/2} \Big(-A_1^{-1/2} A_0 A_1^{-1/2}\Big) A_1^{1/2}$, where
$-A_1^{-1/2} A_0 A_1^{-1/2}$ is a Hermitian matrix, which is diagonalizable
through a unitary matrix, say, $U$. Hence, $C = A_1^{-1/2} U^{-1} D U A_1^{1/2}$, where
$D$ is a diagonal matrix. Thus $C$ is diagonalizable through matrix $X = U A_1^{1/2}$,
where $A_1^{1/2}$ is a Hermitian positive definite matrix. Thus $\kappa(X) = 
\kappa(A_1^{1/2})$. Since $A_1^{1/2}$ is Hermitian positive definite, 
$\kappa(A_1^{1/2}) = \displaystyle \frac{\lambda_{\text{max}}}{\lambda_{\text{min}}}$, 
where $\lambda_{\text{max}}$ and $\lambda_{\text{min}}$ are respectively the maximum and the minimum eigenvalues of $A_1^{1/2}$. Therefore $\kappa(X) = 
\displaystyle \frac{\lambda_{\text{max}}}{\lambda_{\text{min}}}$.

\medskip
\subsubsection{Condition number of matrix $X$ obtained in Theorem \ref{HW-thm-2}}
\hspace*{\fill}

\medskip

We again discuss two cases with reference to Theorem \ref{doubly-stochastic-thm-5}.

\begin{itemize}
\item [(1)] If $A_2 = A_1 = A_0 = I$ or $A_2 = A_1 = A_0 =I^{\prime}$ where $I$ is 
the identity matrix and $I^{\prime} = 
\begin{bmatrix}
0 & 1 \\
1 & 0
\end{bmatrix}$, the coefficients are unitary matrices and this case reduces to the one 
discussed above. Thus, $\kappa(X) < 2$.

\item[(2)] In the general case, since the coefficients are of size $2\times 2$ one can 
easily verify that $v_1 = \begin{bmatrix}
\frac{-1+i\sqrt{3}}{2} & \frac{-1+i\sqrt{3}}{2} & 1 & 1
\end{bmatrix}^t , v_2 = \begin{bmatrix}
\frac{2i}{\sqrt{3}+i} & \frac{2i}{\sqrt{3}+i} & 1 & 1
\end{bmatrix}^t, v_3=\begin{bmatrix}
\frac{2}{(2a-1)+\sqrt{4a^2-4a-8b+5}} & -\frac{2}{(2a-1)+\sqrt{4a^2-4a-8b+5}} & -1 & 1
\end{bmatrix}^t$ and \\
$v_4 =\begin{bmatrix}
\frac{2}{(2a-1)-\sqrt{4a^2-4a-8b+5}} & \frac{2}{(1-2a)+\sqrt{4a^2-4a-8b+5}} & -1 & 1
\end{bmatrix}^t$ are linearly independent eigenvectors of the block companion matrix $C$.
We can thus choose $X = \begin{bmatrix}
v_1 & v_2 & v_3 & v_4
\end{bmatrix}$.
\end{itemize}

\section{Concluding Remarks}

We have discussed diagonalizability of the block companion matrix of matrix polynomials 
under certain assumptions on the coefficient matrices. As a consequence, we derive the 
Hoffman-Wielandt type inequality for the block companion matrices of such matrix 
polynomials. Besides these, the spectral condition number of the matrix $X$ which appears 
in these inequalities is also estimated. 

\medskip
\noindent
{\bf Acknowledgements:} Pallavi .B and Shrinath Hadimani acknowledge the Council of 
Scientific and Industrial Research (CSIR) and the University Grants Commission (UGC), 
Government of India, for financial support through research fellowships.

\bibliographystyle{amsplain}

\end{document}